\documentclass[11pt,reqno]{amsart}
\usepackage{amsmath, mathtools}
\usepackage{amsthm}
\usepackage{amssymb}
\usepackage{enumerate}
\usepackage{amscd}
\usepackage{color}
\usepackage{pb-diagram}
\usepackage{graphicx}
\usepackage[all, cmtip]{xy}
\usepackage{soul}
\usepackage[colorinlistoftodos]{todonotes}
\usepackage{esint}
\usepackage{hyperref}
\hypersetup{pdftex,colorlinks=true,allcolors=blue}
\usepackage{hypcap}
\usepackage[titletoc]{appendix}
\usepackage{comment}
\usepackage{float}
\restylefloat{table}
\usepackage{tikz}
\usetikzlibrary{calc,decorations.pathmorphing,shapes}

\newcounter{sarrow}

\usepackage{amsrefs}

\theoremstyle{plain}
\newtheorem{lemma}{Lemma}
\newtheorem{prop}[lemma]{Proposition}
\newtheorem{coro}[lemma]{Corollary}

\theoremstyle{definition}
\newtheorem{defi}[lemma]{Definition}
\newtheorem{thm-Intro}{Theorem} 
\newtheorem{cor-Intro}{Corollary} 
\newtheorem{exs}{Examples}
\newtheorem{rmk}{Remark}
\numberwithin{equation}{section}

\newcommand{\Ric}{\textup{Ric}}

\newenvironment{red}{\relax\color{red}}{\relax}
\newenvironment{blue}{\relax\color{blue}}{\hspace*{.5ex}\relax}

\newcommand{\ber}{\begin{red}}
\newcommand{\er}{\end{red}}
\newcommand{\beb}{\begin{blue}}
\newcommand{\eb}{\end{blue}}

\evensidemargin- 0.2 true in

\begin{document}
	
	\title[Bonnet-Myers theorem on K\"ahler and quaternionic K\"ahler manifolds]{Diameter theorems on K\"ahler and quaternionic K\"ahler manifolds under a positive lower curvature bound}
	
	\author[Maria Gordina]{Maria Gordina{$^{\dag}$}}
	\thanks{\footnotemark {$\dag$} Research was supported in part by NSF Grant DMS-1954264.}
	\address{$^{\dag}$ Department of Mathematics\\
		University of Connecticut\\
		Storrs, CT 06269,  U.S.A.}
	\email{maria.gordina@uconn.edu}

	\author{Gunhee Cho}
	\address{Department of Mathematics\\
		University of California, Santa Barbara\\
		Santa Barbara, CA 93106}
	\email{gunhee.cho@math.ucsb.edu}

	

	\date{\today \ \emph{File:\jobname{.tex}}}
	
	\begin{abstract}
		We define the orthogonal Bakry-\'Emery tensor as a generalization of the orthogonal Ricci curvature, and then study diameter theorems on K\"ahler and quaternionic K\"ahler manifolds under positivity assumption on the orthogonal Bakry-\'Emery tensor. Moreover, under such assumptions on the orthogonal Bakry-\'Emery tensor and the holomorphic or quaternionic sectional curvature on a K\"ahler manifold or a quaternionic K\"ahler manifold respectively, the Bonnet-Myers type diameter bounds are sharper than in  the Riemannian case.
	\end{abstract}
	
	\maketitle

	\tableofcontents

	\section{Introduction}
	
	The Bonnet–Myers theorem in \cite{Myers1941} is a fundamental theorem connecting compactness and upper bounds on the diameter of a complete Riemannian manifold. A typical result relies on the assumption of a positive lower bound of the Ricci curvature,  or a positivity assumption of the Bakry-\'Emery Ricci tensor  such as \cite{AndersonM1990, Gromov1981a, Otsu1991, BakryQian2000, WeiWylie2009, Limoncu2012, Limoncu2010}. Such results connecting Ricci curvature or its generalizations and Bonnet-Myers type theorems have been studied for Riemannian manifolds, and on sub-Riemannian manifolds in \cite{BaudoinGarofalo2017, BaudoinGrongKuwadaThalmaier2019}. Our focus is on the setting of K\"ahler geometry and quaternionic K\"ahler geometry for which such results have not been considered so far.

	We will rely on the decomposition of Ricci curvature on K\"ahler manifolds into orthogonal Ricci curvature and holomorphic sectional curvature as follows.
	\begin{equation*}
	\Ric(X,\overline{X})=\Ric^{\perp}(X,\overline{X})+R(X,\overline{X},X,\overline{X})/|X|^4,
	\end{equation*}
	where  $X$ is  a $(1,0)$-tangent vector of the holomorphic tangent bundle on a K\"ahler manifold $M^n$. A similar decomposition holds on a quaternionic K\"ahler geometry. Precise definitions are given by \eqref{eq:hsc} and \eqref{eq:decomposition_of_Ricci_quaternionic}.  Using such a decomposition, it is natural to consider a Bakry-\'Emery tensor in both K\"ahler geometry and quaternionic K\"ahler geometry  similarly to how the Ricci curvature is replaced by a Bakry-\'Emery tensor in Riemannian geometry. The goal of this paper is to study diameter theorems for compact K\"ahler manifolds and quaternionic K\"ahler manifolds under various notions of positive lower bounds on the orthogonal Bakry-\'Emery type tensor corresponding to the orthogonal Ricci curvature.
	
	First, we note that replacing the positivity of the orthogonal Ricci curvature with a weaker notion of positivity is justified by the following observation.  Indeed, the class of complete K\"ahler manifolds with a positive orthogonal Ricci curvature is rather small although an orthogonal Ricci curvature bound is usually weaker than a Ricci curvature bound.
	For example, a complete K\"ahler manifold $M^n, n\geqslant 2$ with a positive lower bound on the orthogonal Ricci curvature must be compact and always projective \cite[Theorem 1.7]{NiZheng2018}.  Moreover, for $n=2$, a compact $M^2$ which admits a K\"ahler metric with  $\Ric^{\perp}>0$ must be biholomorphic to the two-dimensional complex projective space $\mathbb{P}_{\mathbb{C}}^2$, and for $n=3$, a compact K\"ahler manifold under $\Ric^{\perp}>0$ must be biholomorphic to either $\mathbb{P}_{\mathbb{C}}^{3}$ or the smooth quadratic hypersurface in  $\mathbb{P}_{\mathbb{C}}^{4}$ as pointed out by \cite[Theorem 1.8]{LeiQingsongFangyang2018}. On the other hand, as there is an example of complete non-compact Riemannian manifold with a non-negative Ricci curvature lower bound,  there are also complete non-compact examples of K\"ahler manifolds with $\Ric^{\perp}>0$ \cite[p151]{NiZheng2018}. Therefore in order to consider the complete K\"ahler manifolds of the wide class rather than the complete K\"ahler manifolds of the limited classes possible under the positive orthogonal Ricci curvature, it is reasonable to make at least weaker the positivity condition than the orthogonal Ricci curvature. On the other hand, similarly to the Bakry-\'Emery tensor in the Riemannian case,  a complete K\"ahler manifold need not be compact even under if the orthogonal Bakry-\'Emery tensor satisfies a positive lower bound as shown in Example~\ref{example}.
	
	
	As the Ricci curvature can be written as a sum of orthogonal Ricci curvature and holomorphic (respectively, quaternionic) sectional curvature for K\"ahler (respectively, quaternionic K\"ahler) manifolds, we consider two versions of an orthogonal Bakry-\'Emery tensor corresponding to the orthogonal Ricci curvature. In the first case we consider the orthogonal Bakry-\'Emery tensor $\Ric^{\perp}+\operatorname{Hess}(\phi)$, where $\phi$  is a real-valued smooth function  on $M$ and $\operatorname{Hess}(\phi)$ is the Riemannian Hessian. That is, we omit the holomorphic (respectively, quaternionic) sectional curvature. Note that here we only consider the Bakry-\'Emery  tensor of a gradient type $\operatorname{Hess}(\phi)$. Similarly to how  compactness and diameter have been treated on complete Riemannian manifolds in \cite{Limoncu2012,Limoncu2010}, here too, additional  smoothness assumptions are needed to obtain the results.
	
	Another approach to compactness and diameter bounds' results using Ricci curvature or  Bakry-\'Emery tensor assumptions on a Riemannian manifold is to rely on Bochner's formula. We are exploring such an approach for complete K\"ahler manifolds with an orthogonal Ricci curvature bound or its generalization. For this purpose, we derive a new Bochner's formula in Proposition~\ref{prop:5} for the orthogonal Ricci curvature, and then establish such results  under the assumptions which are compatible with this Bochner's formula. The second case is to consider a non-gradient type Bakry-\'Emery type tensor $\Ric^{\perp}_{m,Z}$ defined by \eqref{BE-non-gradient} with a vector field $Z$  and an additional assumption on the holomorphic sectional curvature (quaternionic sectional curvature in the case of quaternionic K\"ahler manifolds). In the previous case, the second-order differential operator in  Bochner's formula for the orthogonal Ricci curvature is not hypoelliptic in general, whereas in the second case, as in the Riemannian case, we use the Laplace-Beltrami operator making it possible to use a weaker positivity assumption than in the first case when we replaced the orthogonal Ricci curvature by a Bakry-\'Emery tensor. In order to show the diameter upper bound, we follow Kuwada's approach in \cite{Kuwada2013a} and consider a stochastic process with this operator as a generator that might be non-symmetric. We then prove an upper bound on the diameter which is sharper than the diameter upper bound in the Riemannian case.
	
	This paper is organized as follows. In Section \ref{s.Definition} we introduce basic definitions and properties of K\"ahler manifolds and quaternionic K\"ahler manifolds, and in particular, how these structures are connected  to their Riemannian structures. In Section \ref{s.BakryEmeryOrth}, diameter theorems are covered under the Bakry-\'Emery orthogonal Ricci tensor of the gradient type. In Section \ref{s.BakryEmeryNonSymmetric} we prove diameter theorems for a non-gradient Bakry-\'Emery tensor under the additional assumption on  the holomorphic (quaternionic) sectional curvature.

	\section{Preliminaries on K\"ahler and quaternionic K\"ahler manifolds}\label{s.Definition}
	We start by reviewing basics of K\"ahler and quaternionic manifolds.
	\subsection{K\"ahler manifolds}
	
	Let $M$ be an $n$-dimensional complex manifold equipped with a complex structure $J$ and a Hermitian metric $g$. The complex structure $J : T_{\mathbb{R}}M \rightarrow T_{\mathbb{R}}M$ is a real linear endomorphism that satisfies for every $x \in M$, and $X,Y \in T_{\mathbb{R},x} M$, $g_x(J_x X,Y)=-g_x(X,J_xY)$, and for every $x \in M$, $J_x^2=-\mathbf{Id}_{T_x M} $. We decompose the complexified tangent bundle  $T_{\mathbb{R}}M\otimes_{\mathbb{R}}\mathbb{C}=T'M\oplus \overline{T'M}$, where $T'M$ is the eigenspace of $J$ with respect to the eigenvalue $\sqrt{-1}$ and $\overline{T'M}$ is the eigenspace of $J$ with respect to the eigenvalue $-\sqrt{-1}$. We can identify $v, w$ as real tangent vectors, and $\eta,\xi$ as corresponding holomorphic $(1,0)$ tangent vectors under the $\mathbb{R}$-linear isomorphism $T_{\mathbb{R}}M \rightarrow T'M$, i.e. $\eta=\frac{1}{{\sqrt{2}}}(v-\sqrt{-1}Jv),, \xi=\frac{1}{{{\sqrt{2}}}}(w-\sqrt{-1}Jw)$.
	
	A {Hermitian metric} on $M$ is a positive definite Hermitian inner product
	\[
	g_p : T'_p M \otimes \overline{T'_p M} \rightarrow \mathbb{C}
	\]
	which varies smoothly for each $p \in M$. Here, varying smoothly means that if $z=(z_1,\cdots,z_n)$ are local coordinates around $p$, and $\frac{\partial}{\partial z_1},\cdots,\frac{\partial}{\partial z_n}$ is a standard basis for $T'_p M$, the functions \[ g_{i\overline{j}} : U \rightarrow \mathbb{C}, p \mapsto g_p(\frac{\partial}{\partial z_i},\frac{\partial}{\partial \overline{z_j}} ) \] are smooth for all $i,j=1,\cdots,n$. Locally, a Hermitian metric can be written as
	\[
	g=\sum_{i,j=1}^{n}g_{i\overline{j}}dz_i \otimes d\overline{z_j},
	\]
	where $(g_{i\overline{j}})$ is an $n\times n$ positive definite Hermitian matrix of smooth functions and $dz_1,\cdots,dz_n$ be the dual basis of $\frac{\partial}{\partial z_1},\cdots,\frac{\partial}{\partial z_n}$. The metric $g$ can be decomposed into the real part denoted by $\operatorname{Re}(g)$, and the imaginary part, denoted by $\operatorname{Im}(g)$. $\operatorname{Re}(g)$ induces an inner product called the induced Riemannian metric of $g$, an alternating $\mathbb{R}$-differential $2$-form. Define the $(1,1)$-form $\omega:=-\frac{1}{2}Im(g)$, which is called the fundamental $(1,1)$-form of $g$. In local coordinates this form can written as
	\[
	\omega=\frac{\sqrt{-1}}{2}\sum_{i,j=1}^{n}g_{i\overline{j}}dz_i \wedge d\overline{z_j}.
	\]
	In this setting we have two natural connections. The Chern connection $\nabla^{c}$  is compatible with the Hermitian metric $g$ and the complex structure $J$, and the Levi-Civita connection $\nabla$  is a torsion free connection compatible with the induced Riemannian metric. The components of the curvature $4$-tensor of the Chern connection associated with the Hermitian metric $g$ are given by
	\begin{align*}
	& R_{i\overline{j}k\overline{l}}:= R(\frac{\partial}{\partial z_i},\frac{\partial}{\partial z_i},\frac{\partial}{\partial z_i},\frac{\partial}{\partial z_i})
	\\
	& =g \left( \nabla^c_{\frac{\partial}{\partial z_i} } \nabla^c_{\frac{\partial}{\partial \overline{z_j}} }\frac{\partial}{\partial z_k} -\nabla^c_{\frac{\partial}{\partial \overline{z_j}} } \nabla^c_{\frac{\partial}{\partial z_i}}\frac{\partial}{\partial z_k} -\nabla^c_{[{\frac{\partial}{\partial z_i} },{\frac{\partial}{\partial \overline{z_j}} }]} {\frac{\partial}{\partial z_k} }, {\frac{\partial}{\partial \overline{z_l}}}\right)
	\\
	& =-\frac{\partial^2 g_{i\overline{j}}}{\partial z_k \partial \overline{z}_l}+\sum_{p,q=1}^{n}g^{q\overline{p}}\frac{\partial g_{i\overline{p}}}{\partial z_k}\frac{\partial g_{q\overline{j}}}{\partial \overline{z}_l},
	\end{align*}
	where $i,j,k,l\in \left\{1.\cdots, n \right\}$.
	
	The Hermitian metric $g$ is called K\"ahler if $d\omega=0$, where $d$ is the exterior derivative $d=\partial+\overline{\partial}$, and the Chern and Levi-Civita  connections coincide precisely when the Hermitian metric is K\"ahler. There are several equivalent ways to show that a metric is K\"ahler, and one of them is that a metric $g$ is K\"ahler if and only if for any $p\in M$, there exist  holomorphic coordinates $(z_1,\cdots,z_n)$ near $p$ such that $g_{i\overline{j}}(p)=\delta_{i\overline{j}}$ and $(d g_{i\overline{j}})(p)=0$. Such coordinates are called \emph{holomorphic normal coordinates}.
	
	The holomorphic sectional curvature with the unit direction $\eta$ at $x \in M$ (i.e., $g_{\omega}(\eta,\eta)=1$) is defined by
	\begin{equation*} H(g)(x,\eta)=R(\eta,\overline{\eta},\eta,\overline{\eta})=R(v,Jv,Jv,v),
	\end{equation*}
	where $v$ is the real tangent vector corresponding to $\eta$. We will often write $H(g)(x,\eta)=H(g)(\eta)=H(\eta)$.
	
	Following \cite{NiZheng2019} we define the \emph{orthogonal Ricci curvature}  on a K\"ahler manifold $(M,g,J)$  by
	\begin{equation}\label{eq:hsc}
	\operatorname{Ric}^\perp (v,v)=\operatorname{Ric} (v,v)-H(v),
	\end{equation}
	where  $v$ is a real vector field  and $\mathrm{Ric}$ is the Ricci $2$-tensor of $(M,g)$. Unlike the Ricci tensor, $\operatorname{Ric}^{\perp}$ does not admit  polarization, so we never consider $\operatorname{Ric}^\perp (u, v)$ for $u\not= v$. For a real vector field $v$, we can write
	\begin{equation*}
   {Ric}^\perp(v,v)=\sum R(v,E_i,E_i,v),
	\end{equation*}
	where $\left\{e_i \right\}$ is any orthonormal frame of $\left\{v,Jv \right\}^{\perp}$. We will assign index $1,2$ to $v$ and $Jv$ in this summation expression for complex $n$ dimensional K\"ahler manifold $M^n$, and use indices from $3$ to $2n$ for orthonormal frames $\left\{E_i \right\}$ of $\left\{v,Jv \right\}^{\perp}$. Denote by $F_i=\frac{1}{\sqrt{2}}(E_i-\sqrt{-1}J(E_i))$  a unitary frame such that $E_1=v/|v|=:\widetilde{v}$ by following the convention $E_{n+i}=J(E_i)$, then
	\begin{align*}
	\frac{1}{|v|^2}\mathrm{Ric}^\perp(v,v)&=\mathrm{Ric}^\perp(\widetilde{v},\widetilde{v})=\mathrm{Ric} (\widetilde{v},\widetilde{v})-R(\widetilde{v},J\widetilde{v},\widetilde{v},J\widetilde{v})\\
	&=\mathrm{Ric}(F_1,\overline{F_1})-R(F_1,\overline{F_1},F_1,\overline{F_1})=\sum_{j=2}^{n}R(F_1,\overline{F_1},F_j,\overline{F_j}).
	\end{align*}
	In particular, we have $\mathrm{Ric}(F_i,\overline{F_i})=\mathrm{Ric}(E_i,E_i)$, $\mathrm{Ric}^\perp(\widetilde{v},\widetilde{v})=\mathrm{Ric}(F_1,\overline{F_1})-R_{1\overline{1}1\overline{1}}$.
	
	\subsection{Quaternionic K\"ahler manifolds}
	
	We start by recalling the following definition of quaternionic K\"ahler manifold following \cite[Proposition 14.36]{BesseBook2008}.
	
	\begin{defi}\label{d.QKM}
		A Riemannian manifold $(M, g)$ is called a \emph{quaternionic K\"ahler manifold} if  there exists a covering of $M$ by open sets $U_i$ and, for each $i$ there exist  3 smooth  $(1,1)$ tensors $I,J,K$ on $U_i$ such that
		
		\begin{itemize}
			\item For every $x \in U_i$, and $X,Y \in T_xM$, $g_x(I_x X,Y)=-g_x(X,I_xY)$,  $g_x(J_x X,Y)=-g_x(X,J_xY)$, $g_x(K_x X,Y)=-g_x(X,K_xY)$ ;
			\item For every $x \in U_i$, $I_x^2=J_x^2=K_x^2=I_xJ_xK_x=-\mathbf{Id}_{T_x M} $;
			\item For every $x \in U_i$, and $X\in T_x M$  $\nabla_X I, \nabla_X J, \nabla_X K \in \mathbf{span} \{ I,J,K\}$;
			\item For every $x \in U_i \cap U_j$, the vector space of endomorphisms of $T_x M$ generated by $I_x,J_x,K_x$ is the same for $i$ and $j$.
		\end{itemize}
	\end{defi}
	Note that in some cases such as the quaternionic projective spaces the tensors $I,J,K$ may not be defined globally  for topological reasons. However, $\mathbf{span} \{ I,J,K\}$ may always be defined globally according to Definition~\ref{d.QKM}.
	
	On quaternionic K\"ahler manifolds, we will be considering the following curvature tensors. As above, let
	\begin{equation}\label{eq:Riemann_curvature}
	R(X,Y,Z,W)=g ( (\nabla_X \nabla_Y -\nabla_Y \nabla_X -\nabla_{[X,Y]} )Z, W)
	\end{equation}
	be the Riemannian curvature tensor of $(M,g)$. We define the quaternionic sectional  curvature of the quaternionic K\"ahler manifold $(M,g,J)$ as
	\[
	Q(X)=\frac{R(X,IX,IX,X)+R(X,JX,JX,X)+R(X,KX,KX,X)}{g(X,X)^2}.
	\]
	Following \cite[Section 2.1.2]{BaudoinYang2020} we define the o\emph{rthogonal Ricci curvature}  of the quaternionic K\"ahler manifold $(M,g,I,J,K)$   by
	\begin{equation}\label{eq:decomposition_of_Ricci_quaternionic}
	\mathrm{Ric}^\perp (X,X)=\mathrm{Ric} (X,X)-Q(X),
	\end{equation}
	where $\mathrm{Ric}$ is the usual Riemannian Ricci tensor of $(M,g)$ and $X$ is  a vector field   such that $g(X,X)=1$.
	
	Lastly, given a vector field $V$ on a Riemannian manifold along a geodesic $\gamma$ defined on $[a,b]$, the index form $\mathcal{I}$ associated to $V$ is defined as
	\begin{equation*}
	\mathcal{I}(V,V)=\int_{a}^{b}\left(|\dot{V}(s)|^2-R(V(s),\dot{\gamma}(s),\dot{\gamma}(s),V(s))  \right)ds,
	\end{equation*}
	and using polarization the form $\mathcal{I}$ can be extended to a bilinear form on the space of vector fields along the geodesic $\gamma$.
	
	\section{Bakry-\'Emery orthogonal Ricci tensor of the gradient type}\label{s.BakryEmeryOrth}
	
	Given a Riemannian manifold $(M,g)$ and a smooth function $\phi : M \rightarrow \mathbb{R}$, we denote the Hessian of $\phi$ by $\operatorname{Hess}(\phi)$, i.e., $\operatorname{Hess}(\phi)(X,Y)=g(\nabla_X \nabla \phi, Y)$ for any real vector fields $X,Y$. In this section, we define and consider the orthogonal Bakry-\'Emery tensor $\Ric^{\perp}+\operatorname{Hess}(\phi)$ with a smooth function $\phi$ on either a K\"ahler manifold or a quaternionic K\"ahler manifold. On a K\"ahler manifold $M^n$ with $k\in \mathbb{R}$, we say $\Ric^{\perp}+\operatorname{Hess}(\phi) \geqslant (2n-2)k$ if for any unit vector $v$
	\begin{equation*}
	\Ric^{\perp}(v,v)+\mathrm{\operatorname{Hess}}(\phi)(v,v)\geqslant (2n-2)k,
	\end{equation*}
	and similarly, we assume $\Ric^{\perp}(v,v)+\mathrm{\operatorname{Hess}}(\phi)(v,v)\geqslant (4n-4)k$ for a quaternionic K\"ahler manifold. The Bakry-\'Emery tensor considered in this section is different from such a tensor in Section~\ref{s.BakryEmeryNonSymmetric}. Indeed,  the modified Bochner's formula in Proposition~\ref{prop:5}  shows the relationship between orthogonal Laplacian and orthogonal Ricci curvature, without any assumptions on holomorhpic (or quaternionic) sectional curvature. With the modified Bochner formula, assumptions on a smooth function $\phi$ are important when trying to prove diameter theorems.
	
	Previously Riemannian manifolds endowed with a weighted volume measure $e^{-f}dV_g$ satisfying a lower bound on the standard Bakry-E\'mery Ricci tensor has been studied in several settings. For example, a Riemannian manifold $(M,g)$ is called a gradient Ricci soliton if there exists a real-valued smooth function $f$ on $M$ such that the Ricci curvature and the Hessian of $f$ satisfy $\Ric+\operatorname{Hess}(f)=\lambda g$ for some $\lambda \in \mathbb{R}$. Gradient Ricci solitons play an important role in the theory of Ricci flow as in \cite{Fernandez-LopezGarcia-Rio2008}. Bakry-E\'mery Ricci tensor plays fundamental role in \cite{WeiWylie2009}, and it has been extended to metric measure spaces using the Lott-Villani-Sturm theory initiated by \cite{LottVillani2009, Sturm2006b}. In particular, if $(M,g)$ is a K\"ahler manifold and it is a gradient Ricci soliton with a real-valued smooth function $f$, then $\nabla f$ is a real holomorphic vector field, i.e., its $(1,0)$-part is a holomorphic vector field. Moreover, the weighted Hodge Laplacian from considered with respect to the weighted volume measure $e^{-f}dV_g$ maps the space of smooth $(p,q)$ forms to itself for $0<p+q<2n$ if and only if $\nabla f$ is a real holomorphic vector field by \cite[Proposition 0.1]{OvidiuMunteanuJiapingWang15}. Also, the real holomorphic vector field serves as a critical point of the Calabi functional which yields the Calabi's extremal K\"ahler metric that can be used to study the existence of the K\"ahler-Einstein metric on Fano manifolds \cite{EugenioCalabi82}. Lastly, on a compact K\"ahler manifold with Ricci curvature bounded below by a positive constant $k$, if the first nonzero eigenvalue achieves its optimal lower bound $2k$ then the gradient vector field of the corresponding eigenfunction must be real holomorphic by \cite{Udagawa1988}.
	
	In our case, we are interested in bounds on $\Ric^{\perp}+\operatorname{Hess}(\phi)$, where $\phi$ is a real-valued smooth function, and the holomorphic sectional curvature is not controlled, hence it is different from the case where bounds on the Bakry-\'Emery tensor are used.
	
	For compactness of a complete Riemannian manifold with an upper bound on the diameter,  one needs additional assumptions on the function whose Hessian is used to define the Bakry-\'Emery Ricci tensor as pointed out by \cite{WeiWylie2009}. The next example indicates that we need to address such an issue for the orthogonal Bakry-\'Emery tensor as well.
	
	\begin{exs}\label{example}
		Let $M=\mathbb{C}^n$ with the Euclidean metric $g_E$, and $\phi(x)=\frac{\lambda}{2}|x|^2$. Then both the orthogonal Ricci curvature and the holomorphic sectional curvature are zero, and $\mathrm{\operatorname{Hess}}(\phi)=\lambda g_{E}$ and $(\Ric^{\perp}+\mathrm{\operatorname{Hess}}(\phi))(v,v)=\lambda g_{E}(v,v)$ for any real vector field $v$. This example shows that a complete K\"ahler manifold with positive lower bound of $\Ric^{\perp}+\mathrm{\operatorname{Hess}}(\phi)$ is not necessarily compact.
	\end{exs}
	
	As we observed from the example, to prove that a complete K\"ahler (and also quaternionic K\"ahler) manifold with a positive lower bound on  the orthogonal Bakry-\'Emery tensor is compact we need to have additional assumptions on $\phi$, and we will start with the case when $\phi$ is bounded. In this case, we have the following elementary lemma to control such a function $\phi$. This fact has been used in the proof of \cite[Theorem 1]{Limoncu2012} for a  Bakry-\'Emery Ricci tensor on Riemannian manifolds.
	
	\begin{lemma}\label{Lemma:gradient}
		Let $M$ be a Riemannian manifold with a Riemannian metric $g$ and let $\phi : M \rightarrow \mathbb{R}$ be a smooth function satisfying $|\phi|\leqslant C$ for some $C\geqslant 0$. Let $\gamma$ be a minimizing unit speed geodesic segment from $p$ to $q$ of length $l$. Then we have
		\begin{equation*}
		\int_{0}^{l}f^2\mathrm{\operatorname{Hess}}(\phi) \leqslant 2C\sqrt{l}\left(\int_{0}^{l}(\frac{d}{dt}(f\dot{f})^2)dt\right)^{1/2},
		\end{equation*}	
		for any smooth function $f\in C^{\infty}([0,l])$ such that $f(0)=f(l)=0$. Here $\dot{f}(t)$ means $\frac{d}{dt}f(\gamma(t))$.
		\begin{proof}
			From
			\begin{align*}
			&f(t)^2\mathrm{\operatorname{Hess}}(\phi) ({\dot{\gamma},\dot{\gamma}})(\gamma(t))=f(t)^2\frac{d}{dt}\left(g(\nabla \phi,\dot{\gamma})\right)(\gamma(t))\\
			&=-2f(t)f(t)'g(\nabla \phi,\dot{\gamma})+\frac{d}{dt}\left(f(t)^2 g(\nabla \phi,\dot{\gamma})\right)(\gamma(t))\\
			&=2\phi(\gamma(t)) \frac{d}{dt}(f(t)f(t'))(\gamma(t))-2\frac{d}{dt}(\phi f(t)f(t')(\gamma(t))+\frac{d}{dt}(f(t)^2 g(\nabla \phi,\dot{\gamma}))(\gamma(t)),
			\end{align*}
			and $f(0)=f(l)=0$, we have
			\begin{align*}
			\int_{0}^{l}f^2\mathrm{\operatorname{Hess}}(\phi) ({\dot{\gamma},\dot{\gamma}}))dt=2\int_{0}^{l}\phi \frac{d}{dt}(f \dot{f})dt.
			\end{align*}
			Now the Cauchy-Schwarz inequality with the assumption $|\phi|\leqslant C$ implies
			\begin{equation*}
			\int_{0}^{l}f^2 \mathrm{\operatorname{Hess}}(\phi)=2\int_{0}^{l} \phi \frac{d}{dt}(f\dot{f})dt \leqslant 2C\sqrt{l}\left(\int_{0}^{l}(\frac{d}{dt}(f\dot{f})^2)dt\right)^{1/2}.
			\end{equation*}	
		\end{proof}

	\end{lemma}

	\begin{prop}\label{prop:1}
		Let $(M^n,g)$ be a complete K\"ahler manifold with the complex dimension $n\geqslant 2$. Suppose that for some constant $k>0$, $\Ric^{\perp}+\operatorname{Hess}(\phi) \geqslant (2n-2)k$ and $|\phi|\leqslant C$ for some $C\geqslant 0$. Then the diameter $D$ of $M$ has the upper bound
		\begin{equation*}
		D \leqslant \frac{\pi}{\sqrt{k}}\sqrt{1+\frac{\sqrt{2}C}{n-1}}.
		\end{equation*}
	\end{prop}
	\begin{proof}
		Let $p,q \in M$ and let $\gamma$ be a  minimizing unit speed geodesic segment from $p$ to $q$ of length $l$. Consider a parallel orthonormal frame
		\[
		\left\{ E_1=\dot{\gamma}, E_2=JE_1,\cdots, E_{2n} \right\}
		\]
		along $\gamma$ and a smooth function $f\in C^{\infty}([0,l])$ such that $f(0)=f(l)=0$. Here we used the K\"ahler condition $\nabla J=0$ and parallel transport to have $E_2=JE_1$. From the definition of $\Ric^{\perp}$, we have
		\begin{align*}
		\sum_{i=3}^{2n} \mathcal{I}(fE_i,fE_i)&=\int_{0}^{l}((2n-2)\dot{f}^2-\sum_{i=3}^{2n}R(fE_i,\dot{\gamma},\dot{\gamma}, f E_i)dt\\
		&= \int_{0}^{l}((2n-2)\dot{f}^2-f^2 \Ric^{\perp}(\dot{\gamma},\dot{\gamma})  dt.
		\end{align*}
		By the assumption on the orthogonal Bakry-\'Emery tensor,
		\begin{align*}
		\sum_{i=3}^{2n} \mathcal{I}(fE_i,fE_i)\leqslant \int_{0}^{l}( (2n-2)(\dot{f}^2 - k f^2) +f^2\operatorname{Hess}(\phi) ({\dot{\gamma},\dot{\gamma}}))dt,
		\end{align*}
		where $\mathcal{I}$ denotes the index form of $\gamma$. By Lemma~\ref{Lemma:gradient}, we have
		\begin{equation*}
		\sum_{i=3}^{2n} \mathcal{I}(fE_i,fE_i)\leqslant \int_{0}^{l}( (2n-2)(\dot{f}^2 - k f^2)dt+2C\sqrt{l}\left(\int_{0}^{l}(\frac{d}{dt}(f\dot{f})^2)dt\right)^{1/2}.
		\end{equation*}
		Now, take $f$ to be $f(t)=\sin(\frac{\pi}{l}t)$, then we get
		\begin{align*}
		\sum_{i=3}^{2n} \mathcal{I}(fE_i,fE_i)&
		\leqslant (2n-2)\int_{0}^{l}\left(\frac{\pi^2}{l^2}\cos^2(\frac{\pi}{l}t)-k\sin^2(\frac{\pi}{l}t) \right)dt
		\\
		& +\frac{2C\pi^2}{l\sqrt{l}}\left(\int_{0}^{l}\cos^2(\frac{2\pi}{l}t) dt\right)^{1/2},
		\end{align*}
		and therefore
		\begin{equation*}
		\sum_{i=3}^{2n} \mathcal{I}(fE_i,fE_i)\leqslant -\frac{1}{l}\left((n-1)kl^2-\sqrt{2}C\pi^2-(n-1)\pi^2  \right).
		\end{equation*}
		
		Now if $(n-1)kl^2-\sqrt{2}C\pi^2-(n-1)\pi^2>0$, this forces $\mathcal{I}(fE_m,fE_m)<0$ for some $3\leqslant m \leqslant 2n$. On the other hand, since $\gamma$ is a  minimizing geodesic, the index form $\mathcal{I}$ is positive semi-definite, which is a contradiction. Therefore,
		\begin{equation*}
		l\leqslant \frac{\pi}{\sqrt{k}}\sqrt{1+\frac{\sqrt{2}C}{n-1}}.
		\end{equation*}
		
	\end{proof}
	
	\begin{rmk}
		By taking $C=0$, we can conclude that a complete K\"ahler manifold $M$ with a positive lower bound on the orthogonal Ricci curvature implies  compactness of $M$ and thereby  implies that its  fundamental group is finite by \cite[Theorem 3.2]{NiZheng2018}.
	\end{rmk}
	
	The proof of Proposition~\ref{prop:1} is easily modified to the quaternionic K\"ahler case.
	
	\begin{prop}\label{prop:2}
		Let $(M^n,g,I,J,K)$ be a complete quaternionic K\"ahler manifold of the quaternionic dimension $n\geqslant 2$. Suppose that for some constant $k>0$, $\Ric^{\perp}+\operatorname{Hess}(\phi) \geqslant (4n-4)k$ and $|\phi|\leqslant C$ for some $C\geqslant 0$. Then the diameter $D$ of $M$ satisfies the upper bound
		\begin{equation*}
		D \leqslant \frac{\pi}{\sqrt{k}}\sqrt{1+\frac{\sqrt{2}C}{2n-2}}.
		\end{equation*}
	\end{prop}
	\begin{proof}
		We consider an orthonormal frame $\{ X_1(x),\cdots, X_{4m}(x) \}$ around $x \in M$ such that
		\[
		X_1(x) =\gamma^{\prime}(0), \, X_2(x)=I \gamma^{\prime}(0), \, X_3(x)=J \gamma^{\prime}(0), \, X_4(x)=K \gamma^{\prime}(0)
		\]
		
		We introduce the function
		\[
		\mathfrak{j}(k,t) = \cos  \sqrt{k} t  +\frac{1- \cos  \sqrt{k} t }{ \sin  \sqrt{k} t }   \sin  \sqrt{k} t,
		\]
		and we denote by $X_1,\cdots, X_{4m} $  vector fields obtained by parallel transporting  $X_1(x),\cdots, X_{4m}(x)$ along $\gamma$ and consider  the vector fields defined along $\gamma$ by
		\[
		\tilde{X}_2(\gamma(t))=\mathfrak{j}(4k,t) X_2, \, \tilde{X}_3(\gamma(t))=\mathfrak{j}(4k,t) X_3, \tilde{X}_4(\gamma(t))=\mathfrak{j}(4k,t) X_4
		\]
		and for $i=5,\cdots,4m$ by
		\[
		\tilde{X}_i(\gamma(t))=\mathfrak{j}(k,t) X_i.
		\]
		Then the result follows by arguments similar to the previous proof.
	\end{proof}		
	
	Using an argument similar to the case of a bounded function $\phi$ in the summation of the index form in the proof of the {Proposition \ref{prop:1}}, we can prove similar types of diameter theorems based on different assumptions on $\phi$. Let us just mention one setting that can easily replace Lemma~\ref{Lemma:gradient}. Given a vector field $V$ on a Riemannian manifold $(M,g)$, we denote the Lie derivative of $V$ by $\mathcal{L}_{V}g$ (see \cite{Limoncu2010}).
	\begin{lemma}\label{Lemma:Lie-derivative}
		Let $M$ be a Riemannian manifold with a Riemannian metric $g$ and let $V$ be the smooth vector field satisfying $|V|\leqslant C$ for some $C\geqslant 0$. Let $\gamma$ be a  minimizing unit speed geodesic segment from $p$ to $q$ of length $l$. Then we have
		\begin{equation*}
		\int_{0}^{l}f^2 \mathcal{L}_{V}g(\dot{\gamma},\dot{\gamma})dt \leqslant C\sqrt{l}\left(\int_{0}^{l}f^2\dot{f}^2dt \right)^{1/2},
		\end{equation*}	
		for any smooth function $f\in C^{\infty}([0,l])$ such that $f(0)=f(l)=0$.
		

	\end{lemma}
	
	Now we will consider different assumptions on smooth functions $\phi$ to apply the condition of $\Ric^{\perp}+\operatorname{Hess}(\phi) \geqslant (2n-2)k$ differently on both K\"ahler manifolds or quaternionic K\"ahler manifolds based on the modified Bochner type formula. The Bochner formula was used to control the Laplace-Beltrami operator under the Bakry-\'Emery tensor with the certain assumption on $\phi$ \cite{Limoncu2012}, and we modify this approach. To do so, we first define the orthogonal Laplacian $\triangle^{\perp}$ on K\"ahler manifolds (quaternion K\"ahler case would be similar) as follows: given any fixed point $p$ on a complete K\"ahler manifold $(M^n,J)$ of complex dimension $n$. For each real-valued smooth function $f$, consider the holomorphic vector field $Z=\frac{1}{\sqrt{2}}(\nabla f-\sqrt{-1}J(\nabla f))$ corresponding to the real vector field $\nabla f$ and define
	\begin{equation*}
	\triangle^{\perp}f:= \triangle f - \operatorname{Hess}(f)(Z,\overline{Z}).
	\end{equation*}
	(also see  \cite[p.151]{NiZheng2018})
	Equivalently, with $(E_i)^{2n}_{i=1}$ be an orthonormal frame with $E_1=\nabla f$ and $E_2=JE_1$,
	\begin{equation*}
	\triangle^{\perp} f=\triangle f -\operatorname{Hess}(f)(E_1,E_1) - \operatorname{Hess}(f)(JE_1,JE_1)=\sum_{i=3}^{2n}\operatorname{Hess}(f)(E_i,E_i).
	\end{equation*}
	Similarly, the orthogonal Laplacian of a real-valued smooth function $f$ on a quaternionic K\"ahler manifold $(M,I,J,K)$ is defined by
	\begin{align*}
	\triangle^{\perp}f:&=\triangle f -\operatorname{Hess}(f)(E_1,E_1) - \operatorname{Hess}(f)(IE_1,IE_1)
	\\
	&
	-\operatorname{Hess}(f)(JE_1,JE_1)-\operatorname{Hess}(f)(KE_1,KE_1)\\
	&=\sum_{i=5}^{2n}\operatorname{Hess}(f)(E_i,E_i),
	\end{align*}
	where $(E_i)^{2n}_{i=1}$ is an orthonormal frame around $p$ with $E_2=IE_1, E_3=JE_1, E_4=KE_1$.
	
	Since the orthogonal Laplacian is neither a self-adjoint operator nor a hypo-elliptic operator in general, it might be difficult to study this operator and its applications. Nevertheless, the modified Bochner type formula corresponding to the orthogonal Ricci curvature can be established. We hope this form may have several applications on the geometric analysis side in the future. In our paper, we use the following modified Bochner's formula to use in Proposition~\ref{prop:3}. The idea of the proof is a Bochner's formula modified to fit the orthogonal Ricci tensor.
	
	\begin{prop}[Bochner's formula: K\"ahler's case]\label{prop:5}
		Let $(M^n,g,J)$ be a K\"ahler manifold of the complex dimension $n$.  Let $f$ be a real-valued smooth function on $M$ and $\triangle^{\perp}$ be the orthogonal Laplacian. Then for any orthonormal frame $(E_i)^{2n}_{i=1}$ around $q \in M$ with $E_1=\nabla{f}$, $E_2=JE_1$,
		\begin{align*}
		& \frac{1}{2}\sum^{2n}_{i=3}E_i E_i g(\nabla f, \nabla f)(q) =\Ric^{\perp}(\nabla f, \nabla f)(q)+g(\nabla f,\nabla \triangle^{\perp}f)(q)
		\\ \notag
		&
		+\sum^{2n}_{i=3}g(\nabla_{\nabla f} \nabla f,\nabla_{E_i}E_i ))(q) +\sum_{i=3}^{2n}\left(g(\nabla_{E_i} \nabla f,\nabla_{E_i} \nabla f)-2g(\nabla_{\nabla f} {E_i} ,\nabla_{E_i} \nabla f)\right)(q).
		\end{align*}
		
		\begin{proof}
			Let $q\in M$. Then at $q$,
			\begin{align*}
			&\frac{1}{2}\sum^{2n}_{i=3}E_i E_i g(\nabla f, \nabla f)(q)=\sum^{2n}_{i=3}E_i g(\nabla_{E_i} \nabla f,\nabla f)(q)\\
			&=\sum^{2n}_{i=3}E_i \operatorname{Hess}(f)(E_i,\nabla f)(q)=\sum^{2n}_{i=3}E_i \operatorname{Hess}(f)(\nabla f,E_i)(q)\\
			&=\sum^{2n}_{i=3}E_i g(\nabla_{\nabla f}\nabla f, E_i)(q)=\sum^{2n}_{i=3}(g(\nabla_{E_i}\nabla_{ \nabla f}\nabla f,E_i)(q)+g(\nabla_{\nabla f} \nabla f,\nabla_{E_i}E_i ))(q),
			\end{align*}
			By using the Riemann curvature tensor defined by \eqref{eq:Riemann_curvature} we see that
			\begin{align}\label{eq:three-terms}
			&\sum^{2n}_{i=3}(g(\nabla_{E_i}\nabla_{ \nabla f}\nabla f,E_i)(q)\\\nonumber
			&=\sum^{2n}_{i=3}\left(g(R(E_i,\nabla f)\nabla f,E_i)+g(\nabla_{ \nabla f}\nabla_{E_i} \nabla f,E_i )+g(\nabla_{[E_i,\nabla f]} \nabla f, E_i) \right)(q) 
			\\ \nonumber
			&=\Ric^{\perp}(\nabla f,\nabla f)(q)+\sum^{2n}_{i=3}\left( g(\nabla_{ \nabla f}\nabla_{E_i} \nabla f,E_i )+g(\nabla_{[E_i,\nabla f]} \nabla f ,E_i) \right)(q).
			\end{align}
			The second term $\sum^{2n}_{i=3} g(\nabla_{ \nabla f}\nabla_{E_i} \nabla f,E_i )$ in \eqref{eq:three-terms} can be written as
			\begin{align*}
			&\sum^{2n}_{i=3} \left( \nabla f g(\nabla_{E_i}\nabla f,E_i)-g(\nabla_{E_i}\nabla f,\nabla_{\nabla f} E_i ) \right)(q)
			\\
			& =(\nabla f \sum^{2n}_{i=3}g(\nabla_{E_i}\nabla f,E_i))(q)-\sum^{2n}_{i=3}g(\nabla_{E_i}\nabla f,\nabla_{\nabla f} E_i ) \\
			&=\nabla f(\triangle^{\perp}f)-\sum^{2n}_{i=3}g(\nabla_{E_i}\nabla f,\nabla_{\nabla f} E_i )=g(\nabla f,\nabla \triangle^{\perp}f)-\sum^{2n}_{i=3}g(\nabla_{E_i}\nabla f,\nabla_{\nabla f} E_i ).
			\end{align*}
			The last term $g(\nabla_{[E_i,\nabla f]} \nabla f ,E_i) (q)$ in in \eqref{eq:three-terms} can be written as
			\begin{align*}
			&\sum^{2n}_{i=3} \operatorname{Hess}(f)([E_i,\nabla f],E_i)(q)=\sum^{2n}_{i=3} \operatorname{Hess}(f)(\nabla_{E_i}\nabla f-\nabla_{ \nabla f} E_i ,E_i)(q)\\
			&=\sum_{i=3}^{2n}\left(g(\nabla_{E_i} \nabla f,\nabla_{E_i} \nabla f)-g(\nabla_{\nabla f} {E_i} ,\nabla_{E_i} \nabla f)\right)(q).
			\end{align*}
			and we obtain the desired formula.
		\end{proof}
	\end{prop}
	
	\begin{rmk}
		To compare with the usual Bochner's formula, the formula in Proposition~\ref{prop:5} uses the orthogonal Ricci tensor, the orthogonal Laplacian, and
		\begin{equation}\label{eq:3.2}
		\sum^{2n}_{i=3}\left(g(\nabla_{E_i} \nabla f,\nabla_{E_i} \nabla f)-2g(\nabla_{\nabla f} {E_i}, \nabla_{E_i} \nabla f)+g(\nabla_{\nabla f} \nabla f,\nabla_{E_i}E_i )\right)
		\end{equation}
		instead of the Ricci tensor, the Laplace-Beltrami operator, and the Hessian norm squared of $f$ respectively. If we choose an orthonormal frame satisfying $\nabla_{E_i}E_j (q)=0$ at some fixed point $q$ for any $i,j=1,\cdots,2n$, then each of three terms in ~\eqref{eq:3.2} are zero at $q$. However, this does not imply that the first two terms in ~\eqref{eq:3.2} vanish somewhere. For example, if we take $f$ to be the geodesic distance $r$ emanating from $q\in M$, then outside of the cut-locus of $q$, one can see that $\lim_{r\rightarrow 0+} r\triangle^{\perp} r=2n-2$, but $\sum_{i=3}^{2n}g(\nabla_{E_i} \nabla r,\nabla_{E_i} \nabla r)\geqslant  \frac{1}{2n-2}(\triangle^{\perp}r)^2$ (see the proof of Proposition~\ref{prop:3}), and by combining these two, $\sum_{i=3}^{2n}g(\nabla_{E_i} \nabla r,\nabla_{E_i} \nabla r)$ cannot vanish in a small neighborhood of $q$.
	\end{rmk}
	
	With similar computations, one can obtain the quaternionic version of the modified Bochner's formula.
	\begin{prop}[Bochner's formula: quaternionic case] \label{prop:6}
		Let $(M^n,g,I,J,K)$ be a quaternionic K\"ahler manifold of the quaternionic dimension $n$. Let $f$ be a real-valued smooth function on $M$ and $\triangle^{\perp}$ be the orthogonal Laplacian. Then for any orthonormal frame $(E_i)^{4n}_{i=1}$ around $q$ with $E_2=IE_1, E_3=J E_1, E_4=K E_1$,
		\begin{align*}
		\frac{1}{2}\triangle^{\perp}|\nabla f|^2 (q)&=\Ric^{\perp}(\nabla f, \nabla f)(q)+g(\nabla f,\nabla \triangle^{\perp}f)(q)+\sum^{4n}_{i=5}g(\nabla_{\nabla f} \nabla f,\nabla_{E_i}E_i ))(q)\\
		&+\sum_{i=5}^{4n}\left(g(\nabla_{E_i} \nabla f,\nabla_{E_i} \nabla f)-g(\nabla_{\nabla f} {E_i} ,\nabla_{E_i} \nabla f)\right)(q).
		\end{align*}
	\end{prop}	
	
	By Bochner's formula in Proposition~\ref{prop:5} applied to the function $f$ being equal to the geodesic distance $r$ emanating from a point $p\in M$ outside of the cut-locus of $p$, although all terms $g(\nabla_{\nabla r} \nabla r,\nabla_{E_i}E_i ),i=3,\cdots, 2n$  vanish, we still need to control the term
	\[
	-2\sum_{i=3}^{2n}g(\nabla_{\nabla r} {E_i}, \nabla_{E_i} \nabla r)=-2\sum_{i=3}^{2n}\operatorname{Hess}(r)(\nabla_{\nabla r}E_i,E_i ).
	\] 
	With this consideration, it would be natural to compensate the averaging effect $\sum_{i=3}^{2n}\operatorname{Hess}(r)(\nabla_{\nabla r}E_i,E_i )$ by the Hessian of some function $\phi$, for example,
	\begin{equation*}
	\operatorname{Hess}(\phi)(\nabla r, \nabla r)\geqslant 2\sum_{i=3}^{2n}\operatorname{Hess}(r)(\nabla_{\nabla r}E_i,E_i ).
	\end{equation*}
	
	By adding $\operatorname{Hess}(\phi)$ term to the orthogonal Ricci curvature with certain assumptions on $\phi$, we obtain the following diameter theorem. One can see that the assumptions on $\phi$ in the Proposition below are analogous to the assumptions of Theorem 2 in \cite{Limoncu2012} If we replace the complete K\"ahler manifold with $\Ric^{\perp}+\mathrm{\operatorname{Hess}}(\phi)-2\sum_{i=3}^{2n}\operatorname{Hess}(r)(\nabla_{\nabla r}E_i,E_i )$ by the Riemannian manifold with $\Ric+\operatorname{Hess}(\phi)$.
	
	\begin{prop}\label{prop:3}
		Let $(M^n,g)$ be a complete K\"ahler manifold with the complex dimension $n\geqslant 2$. Take any $p\in M$ and let $r$ be the geodesic distance function from $p$. Suppose that for some constant $k>0$, there exists a local orthonormal frame $(E_i)^{2n}_{i=1}$ around $p$ with $E_1=\nabla{r}$, $E_2=JE_1$ such that  
		\[
		\Ric^{\perp}(\nabla r, \nabla r)+\operatorname{Hess}(\phi)(\nabla r, \nabla r)-2\sum_{i=3}^{2n}\operatorname{Hess}(r)(\nabla_{\nabla r}E_i,E_i ) \geqslant (2n-2)k
		\]
		outside of the cut-locus of $p$ and $|\nabla \phi|^2\leqslant \frac{C}{r(x)}$ for some $C\geqslant 0$. Then the diameter $D$ of $M$ satisfies
		\begin{equation*}
		D \leqslant \frac{\pi}{\sqrt{(n-1)k}}\sqrt{2\sqrt{C}+n-1}.
		\end{equation*}	
		
		\begin{proof}
			Define the modified orthogonal Laplacian
			\begin{equation}\label{eq:31}
			\tilde{\triangle}^{\perp} f:=\triangle^{\perp}f-g(\nabla \phi, \nabla f)+F(f),
			\end{equation}
			where $f$ is a smooth function on $M$ and $F$ is a real-valued function taking values from $\mathbb{R}$. We will take $f=r$, here $r$ is the distance function from the fixed point $p \in M$. Then outside of the cut-locus of $p$,
			\begin{equation}\label{eq:32}
			g(\nabla r, \nabla \tilde{\triangle}^{\perp}r)=g(\nabla r, \nabla \tilde{\triangle}^{\perp}r)-\operatorname{Hess}(\phi)(\nabla r, \nabla r)+F^{\prime}(r), F^{\prime}(r)=\frac{d}{dr}F(r).
			\end{equation}
			On the other hand, from Proposition~\ref{prop:5}, since $\frac{1}{2}\triangle^{\perp}|\nabla r|^2 \equiv 0$,
			\begin{equation*}
			0=\Ric^{\perp}(\nabla r, \nabla r)+g(\nabla r,\nabla \triangle^{\perp}r)+\sum_{i=3}^{2n}\left(g(\nabla_{E_i} \nabla r,\nabla_{E_i} \nabla r)+g(\nabla_{\nabla r} {E_i},\nabla_{E_i} \nabla r)\right).
			\end{equation*}
			From the Cauchy-Schwarz inequality, the term $\sum_{i=3}^{2n}g(\nabla_{E_i} \nabla r,\nabla_{E_i} \nabla r)$ has the following lower bound:
			\begin{align*}
			&\sum_{i=3}^{2n}g(\nabla_{E_i} \nabla r,\nabla_{E_i} \nabla r)=\sum_{i=3}^{2n}\sum_{j=1}^{2n} g(\nabla_{E_i} \nabla r,E_j)^2\geqslant \sum_{i=3}^{2n} g(\nabla_{E_i} \nabla r,E_i)^2\\
			&=\frac{1}{2n-2} \sum_{i=3}^{2n} g(\nabla_{E_i} \nabla r,E_i)^2\sum_{i=3}^{2n}1   \geqslant \frac{1}{2n-2}\left(\sum_{i=3}^{2n} g(\nabla_{E_i} \nabla r,E_i)\right)^2= \frac{1}{2n-2}(\triangle^{\perp}r)^2,
			\end{align*}
			thus we have
			\begin{equation}\label{eq:modified-Bochner}
			0\geqslant \Ric^{\perp}(\nabla r, \nabla r)+g(\nabla r,\nabla \triangle^{\perp}r)+\frac{1}{2n-2}(\triangle^{\perp}r)^2,
			\end{equation}
			outside of the cut-locus of $p$.
			
			From \eqref{eq:modified-Bochner} and \eqref{eq:32},
			\begin{equation}\label{eq:34}
			0\geqslant \Ric^{\perp}(\nabla r, \nabla r)+\operatorname{Hess}(\phi)(\nabla r, \nabla r)+g(\nabla r,\nabla {\triangle}^{\perp}r)+\frac{1}{2n-2}({\triangle}^{\perp}r)^2,
			\end{equation}
			Combining with \eqref{eq:31}, \eqref{eq:34} becomes
			\begin{equation}\label{eq:36}
			0\geqslant \Ric^{\perp}(\nabla r, \nabla r)+\operatorname{Hess}(\phi)(\nabla r, \nabla r)+g(\nabla r,\nabla {\triangle}^{\perp}r)+\frac{1}{2n-2}(\tilde{\triangle}^{\perp}r+g(\nabla \phi,\nabla r)-F^{\prime}(r))^2.
			\end{equation}
			From an elementary inequality $(a\mp b)^2 \geqslant \frac{1}{\gamma+1}a^2-\frac{1}{\gamma}b^2$ for any real numbers $a,b$ and $\gamma>0$,
			\begin{equation*}
			(\tilde{\triangle}^{\perp}r+g(\nabla \phi,\nabla r)-F^{\prime}(r))^2\geqslant \frac{1}{\gamma+1}(\tilde{\triangle}^{\perp}r+g(\nabla \phi,\nabla r))^2-\frac{1}{\gamma}(F(r))^2.
			\end{equation*}
			By using the same inequality applied to $(\tilde{\triangle}^{\perp}r+g(\nabla \phi,\nabla r))^2$, any $\gamma,\eta>0$
			\begin{equation}\label{eq:38}
			(\tilde{\triangle}^{\perp}r+g(\nabla \phi,\nabla r))^2 \geqslant \frac{1}{(\gamma+1)\eta+\gamma+1}(\tilde{\triangle}^{\perp}r)^2-\frac{1}{\gamma}(F(r))^2-\frac{1}{(\gamma+1)\eta}(g(\nabla \phi,\nabla r))^2.
			\end{equation}
			Inserting \eqref{eq:38} into \eqref{eq:36} with $\alpha=(2n-2)\gamma>0,\beta=(2n-2)(\gamma+1)\eta>0$,
			\begin{align}\label{eq:39}
			0\geqslant &\Ric^{\perp}(\nabla r, \nabla r)+H(\phi)(\nabla r, \nabla r)+g(\nabla r,\nabla {\triangle}^{\perp}r)\\
			&+\frac{1}{\alpha+\beta+2n-2}(\tilde{\triangle}^{\perp}r)^2-\frac{1}{\alpha}(F(r))^2-\frac{1}{\beta}(g(\nabla \phi,\nabla r))^2.
			\end{align}
			By the Cauchy-Schwarz inequality,
			\begin{equation*}
			(g(\nabla \phi,\nabla r))^2 \leqslant g(\nabla \phi, \nabla \phi)g(\nabla r,\nabla r)=g(\nabla \phi,\nabla \phi),
			\end{equation*}
			thus \eqref{eq:39} becomes
			\begin{align*}
			0\geqslant &\Ric^{\perp}(\nabla r, \nabla r)+H(\phi)(\nabla r, \nabla r)+g(\nabla r,\nabla {\triangle}^{\perp}r)\\
			&+\frac{1}{\alpha+\beta+2n-2}(\tilde{\triangle}^{\perp}r)^2-\frac{1}{\alpha}(F(r))^2-\frac{1}{\beta}(g(\nabla \phi,\nabla \phi)).
			\end{align*}
			From the assumption on $(g(\nabla \phi,\nabla \phi))$,
			\begin{align*}
			0\geqslant &\Ric^{\perp}(\nabla r, \nabla r)+H(\phi)(\nabla r, \nabla r)+g(\nabla r,\nabla {\triangle}^{\perp}r)\\
			&+\frac{1}{\alpha+\beta+2n-2}(\tilde{\triangle}^{\perp}r)^2-\frac{1}{\alpha}(F(r))^2-\frac{C}{\beta r^2}.
			\end{align*}
			Now we take $\beta=\frac{4C}{\alpha}$ and $F(r)=\frac{\alpha}{r}$, then
			\begin{align*}
			0\geqslant \Ric^{\perp}(\nabla r, \nabla r)+H(\phi)(\nabla r, \nabla r)+g(\nabla r,\nabla {\triangle}^{\perp}r)+\frac{\alpha}{\alpha^2+(2n-2)\alpha+4C}(\tilde{\triangle}^{\perp}r)^2.
			\end{align*}
			By the assumption on $\Ric^{\perp}+H(\phi)$,
			\begin{equation*}
			0\geqslant \partial_r({\tilde{\triangle}}^{\perp}r) +\frac{\alpha}{\alpha^2+(2n-2)\alpha+4C}(\tilde{\triangle}^{\perp}r)^2+(2n-2)k.
			\end{equation*}
			From \eqref{eq:31},
			\begin{align*}
			\lim_{r\rightarrow 0+}r{\tilde{\triangle}}^{\perp}r&=\lim_{r\rightarrow 0+}\left(r\triangle r-rg(\nabla \phi,\nabla r)+\frac{\alpha}{2} \right)\\
			&=2n-2+\frac{\alpha}{2}\leqslant \frac{\alpha^2+(2n-2)\alpha+4C}{\alpha}.
			\end{align*}
			Here, we used $\lim_{r\rightarrow 0+}r\triangle r=\lim_{r\rightarrow 0+} r\triangle^{\perp} r=2n-2$. Thus by the Sturm-Liouville comparison argument,
			\begin{equation}\label{eq:key-estimate}
			\tilde{\triangle}^{\perp}r\leqslant \sqrt{(2n-2)k\left(2n-2+\alpha+\frac{4C}{\alpha}\right)}\cot\left(\frac{\sqrt{\alpha(2n-2)k}}{\sqrt{\alpha^2+(2n-2)\alpha+4C}r}\right).
			\end{equation}
			
			Now we use the contradictory argument which was used in \cite{Zhu1997, Limoncu2012} for $\tilde{\triangle}r$. Let $q\in M$ and let $\gamma$ be a minimizing unit speed geodesic segment from $p$ to $q$. Assume that
			\begin{equation*}
			d(p,q)>\frac{\pi}{\sqrt{(2n-2)k}}\sqrt{4\sqrt{C}+2n-2}.
			\end{equation*}
			Then $\gamma\left(\frac{\pi}{\sqrt{(2n-2)k}}\sqrt{4\sqrt{C}+2n-2} \right)$ must belong to $M$ outside of the cut-locus of $p$. In particular, the distance function $r$ is smooth at this point. Now, the left-hand side of \eqref{eq:key-estimate} is constant, whereas the right-hand side goes to $-\infty$, which yields the contradiction. Hence the diameter $D$ of $M$ must satisfy
			\begin{equation*}
			D \leqslant \frac{\sqrt{\alpha^2+(2n-2)\alpha+4C}r}{\sqrt{\alpha(2n-2)k}}\pi.
			\end{equation*}
			By taking $\alpha=2\sqrt{C}$, we have
			\begin{equation*}
			D \leqslant \frac{\pi}{\sqrt{(n-1)k}}\sqrt{2\sqrt{C}+n-1}.
			\end{equation*}
			
		\end{proof}	
		
	\end{prop}
	
	After one obtains \eqref{eq:modified-Bochner}, Proposition~\eqref{prop:3} can be proven similar to the proof of \cite[Theorem 2]{Limoncu2012}. Also the proof can be also generalized to the quaternionic K\"ahler case.	
	
	\begin{prop}\label{prop:4}
		Let $(M^n,g,I,J,K)$ be a complete quaternion K\"ahler manifold with the quaternionic dimension $n\geqslant 2$. Take any $p\in M$ and let $r$ be the geodesic distance function from $p$. Suppose that for some constant $k>0$, there exists a local orthonormal frame $(E_i)^{4n}_{i=1}$ around $p$ with $E_1=\nabla{r}$, $E_2=IE_1, E_3=JE_1, E_4=KE_1$ satisfying  $\Ric^{\perp}(\nabla r, \nabla r)+\operatorname{Hess}(\phi)(\nabla r, \nabla r)-2\sum_{i=5}^{4n}\operatorname{Hess}(r)(\nabla_{\nabla r}E_i,E_i ) \geqslant (4n-4)k$ outside of the cut-locus of $q$ and $|\nabla \phi|^2\leqslant \frac{C}{r(x)}$ for some $C\geqslant 0$. Then the diameter $D$ of $M$ has the upper bound
		\begin{equation*}
		D \leqslant \frac{\pi}{\sqrt{(n-1)k}}\sqrt{\sqrt{C}+n-1}.
		\end{equation*}
	\end{prop}
	
	\section{Bakry-\'Emery orthogonal Ricci tensor associated with the possibly non-symmetric generator of a diffusion process}\label{s.BakryEmeryNonSymmetric}
	
	In this section we consider Bakry-\'Emery orthogonal Ricci tensor for the orthogonal Ricci curvature while imposing an additional assumption on the holomorphic sectional curvature on K\"ahler manifolds (quaternionic sectional curvature in the case of quaternion K\"ahler manifolds). Let $M$ be a real $n$-dimensional connected smooth manifold with $n\geqslant 2$, either equipped with a K\"ahler structure or a quaternionic K\"ahler structure. Choose a complete Riemannian metric $g$ which is compatible with underlying K\"ahler or quaternionic K\"ahler structure and define $\mathcal{L}:=\triangle+Z$, where $\triangle$ is the Laplace-Beltrami operator associated with $g$ and $Z$ a smooth vector field. We denote the Riemannian distance on $M$ associated with $g$ by $r$. Let us define $(0,2)$-symmetric tensor $(\nabla Z)^{{b}}$ by
	\begin{equation*}
	(\nabla Z)^{{b}}(X,Y):=\frac{1}{2}(\langle \nabla_X Z, Y \rangle+\langle \nabla_Y Z, X \rangle ).
	\end{equation*}
	Given a constant $m \in \mathbb{R}$ with $m \geqslant n$, we define the Bakry-\'Emery type orthogonal Ricci tensor $\Ric^{\perp}_{m,Z}$ by
	\begin{equation}\label{BE-non-gradient}
	\Ric^{\perp}_{m,Z}:=\Ric^{\perp} - (\nabla Z)^{{b}}-\frac{1}{m-n}Z \otimes Z.
	\end{equation}
	In this definition, the convention is that for $m=n$ the vector field $Z\equiv 0$.
	
	Throughout this section, we will assume that with some constant $m,k>0$, $\Ric^{\perp}_{m,Z}\geqslant (2m-2)k$. This condition means that for a smooth curve $\gamma(t)$
	\begin{align*}
	& \Ric^{\perp}_{m,Z}(\dot{\gamma}(t),\dot{\gamma}(t))=\Ric^{\perp}(\dot{\gamma}(t),\dot{\gamma}(t))-(\nabla Z)^{b}(\dot{\gamma}(t),\dot{\gamma}(t))-\frac{1}{m-2n}\langle Z(\gamma(t)),\dot{\gamma}(t) \rangle
	\\
	& \geqslant (2m-2)k.
	\end{align*}
	
	\begin{prop}\label{prop:Kahler-Laplacian_comparison}
		Suppose $(M^n,g,J)$ is a K\"ahler manifold  of the complex dimension $n\geqslant 2$. We denote by $r$ the Riemannian distance on $M$ from $p$ associated with the metric $g$, and by $Cut_p$ the cut-locus of $p$.  Suppose that for some constant $k>0$, $\Ric^{\perp}_{m,Z}\geqslant (2m-2)k$ and $H \geqslant 4k$. Let $x\in M \backslash Cut_p \cup \left\{p\right\}$ with $r(x)<\frac{\pi}{2\sqrt{k}}$. Then
		\begin{equation*}
		\mathcal{L} r(x)\leqslant (m-2)\frac{\mathfrak{s}^{\prime}(k,r(x))}{\mathfrak{s}(k,r(x))}+\frac{\mathfrak{s}^{\prime}(4k,r(x))}{\mathfrak{s}(4k,r(x))},
		\end{equation*}	
		where $\mathfrak{s}(k,t):=\sin{\sqrt{k}t}$.
	\end{prop}
	
	\begin{proof}
		When $m=2n$, we have $Z \equiv 0$ and we can use the same argument as we use below. Thus without loss of generality, we will assume $m>2n$.
		
		Let $ p\in M$ and $x\neq p$ which is not in the cut-locus of $p$. Let $\gamma : [0,r(x)] \rightarrow M$ be the unique arclength parameterized geodesic connecting $p$ to $x$. At $x$, we consider an orthonormal frame $\left\{X_1(x),...,X_{2n}(x) \right\}$ such that
		\begin{equation*}
		X_1(x)=\gamma^{\prime}(r(x)), X_2(x)=JX_1(x).
		\end{equation*}
		Then
		\begin{equation*}
		\mathcal{L}r= \triangle r(x) + Zr(x)=\sum_{i=1}^{2n}\nabla^2 r(X_i(x),X_i(x)) + Zr(x).
		\end{equation*}
		Since $X_1(x)=\gamma^{\prime}(r(x))$, $\nabla^2 r(X_1(x),X_1(x))$ is zero. Now we divide the above sum into three parts: $\nabla^2 r(X_2(x),X_2(x))$, $\sum_{i=3}^{2n}\nabla^2 r(X_i(x),X_i(x))$, and $Zr(x)$.
		
		For $\nabla^2 r(X_2(x),X_2(x))$, since $J$ is parellel and $\gamma$ is a geodesic, the vector field defined along $\gamma$ by $J\gamma^{\prime}$ is parellel. Define the vector field along $\gamma$ by
		\begin{equation*}
		\tilde{X}(\gamma(t))=\frac{\mathfrak{s}(4k,t)}{\mathfrak{s}(4k,r(x))}J\gamma^{\prime}(t),
		\end{equation*}
		where $\mathfrak{s}(k,t):=\sin{\sqrt{k}t}$. From the index lemma,
		\begin{align*}
		\nabla^2 r(X_2(x),X_2(x))&\leqslant \int_{0}^{r(x)}\left(\langle \nabla_{\gamma^{\prime}}\tilde{X},\nabla_{\gamma^{\prime}}\tilde{X} \rangle-\langle R(\gamma^{\prime},\tilde{X})\tilde{X},\gamma^{\prime} \rangle \right)dt\\
		&= \frac{1}{\mathfrak{s}(4k,r(x))^2}\int_{0}^{r(x)}(\mathfrak{s}^{\prime}(4k,t)^2-\mathfrak{s}(4k,t)^2\langle R(\gamma^{\prime},J\gamma^{\prime})J\gamma^{\prime},\gamma^{\prime} \rangle )dt\\
		&\leqslant \frac{1}{\mathfrak{s}(4k,r(x))^2}\int_{0}^{r(x)}(\mathfrak{s}^{\prime}(4k,t)^2-4k \mathfrak{s}(4k,t)^2 )dt\\
		\end{align*}
		Next, let us estimate $\sum_{i=3}^{2n}\nabla^2 r(X_i(x),X_i(x))$. We denote by $\left\{X_3,\cdots,X_{2n} \right\}$ the vector fields along $\gamma$ obtained by parallel transport of $\left\{X_3(x),\cdots,X_{2n}(x) \right\}$. Define the vector fields along $\gamma$ by
		\begin{equation*}
		\tilde{X}_i(\gamma(t))=\frac{\mathfrak{s}(k,t)}{\mathfrak{s}(k,r(x))}X_i(\gamma(t)),i=3,\cdots,2m.
		\end{equation*}
		By the index lemma,
		\begin{align*}
		\sum_{i=3}^{2n}\nabla^2 r(X_i(x),X_i(x))&\leqslant \sum_{i=3}^{2n}\int_{0}^{r(x)}\left(\langle \nabla_{\gamma^{\prime}}\tilde{X}_i,\nabla_{\gamma^{\prime}}\tilde{X}_i \rangle-\langle R(\gamma^{\prime},\tilde{X})_i\tilde{X}_i,\gamma^{\prime} \rangle \right)dt\\
		&=\frac{1}{\mathfrak{s}(k,r(x))^2}\sum_{i=3}^{2n} \int_{0}^{r(x)}(\mathfrak{s}^{\prime}(k,t)^2-\mathfrak{s}(k,t)^2\langle R(\gamma^{\prime},\tilde{X}_i)\tilde{X}_i,\gamma^{\prime} \rangle )dt\\
		&\leqslant \frac{1}{\mathfrak{s}(k,r(x))^2} \int_{0}^{r(x)}((2n-2) \mathfrak{s}^{\prime}(k,t)^2-\mathfrak{s}(k,t)^2 \Ric^{\perp}(\gamma^{\prime},\gamma^{\prime}) )dt.
		\end{align*}
		
		For the last term $Z r(x)$, we have
		\begin{align*}
		Z r(x)&=\langle Z(x),\dot{\gamma}(r(x)) \rangle \frac{\mathfrak{s}(k,r(x))^2}{\mathfrak{s}(k,r(x))^2}-\langle Z(p),\dot{\gamma}(0) \rangle \frac{\mathfrak{s}(k,0)^2}{\mathfrak{s}(k,r(x))^2}\\
		&=\int_{0}^{r(x)}\frac{\partial}{\partial t}(\langle Z(\gamma(t)),\dot{\gamma}(t) \rangle \frac{\mathfrak{s}(k,t)^2}{\mathfrak{s}(k,r(x))^2}dt\\
		&=\int_{0}^{r(x)}\left((\nabla Z)^{b}(\dot{\gamma}(t),\dot{\gamma}(t))\frac{\mathfrak{s}(k,t)^2}{\mathfrak{s}(k,r(x))^2}+2\langle Z(\gamma(t)),\dot{\gamma}(t) \rangle\frac{\mathfrak{s}^{\prime}(k,t)}{\mathfrak{s}(k,r(x))})\frac{\mathfrak{s}(k,t)}{\mathfrak{s}(k,r(x))} \right)dt\\
		&\leqslant \int_{0}^{r(x)}\left((\nabla Z)^{b}(\dot{\gamma}(t),\dot{\gamma}(t))\frac{\mathfrak{s}(k,t)^2}{\mathfrak{s}(k,r(x))^2}+\frac{1}{m-2n} \int_{0}^{r(x)} \langle Z(\gamma(t)),\dot{\gamma}(t) \rangle \frac{\mathfrak{s}(k,t)^2}{\mathfrak{s}(k,r(x))^2} \right)dt\\
		&+(m-2n)\int_{0}^{r(x)}\frac{(\mathfrak{s}^{\prime}(k,t))^2}{\mathfrak{s}(k,r(x))^2} dt.
		\end{align*}
		Here the last inequality follows from the arithmetic-geometric mean inequality. Since $\Ric^{\perp}_{m,Z}(\dot{\gamma}(t),\dot{\gamma}(t))=\Ric^{\perp}(\dot{\gamma}(t),\dot{\gamma}(t))-(\nabla Z)^{b}(\dot{\gamma}(t),\dot{\gamma}(t))-\frac{1}{m-2n}\langle Z(\gamma(t)),\dot{\gamma}(t) \rangle\geq(2m-2)k$ for all $t\in [0,r(x)]$, we obtain
		\begin{align*}
		\mathcal{L} r(x)&\leqslant \int_{0}^{r(x)}(m-2)\frac{(\mathfrak{s}^{\prime}(k,t))^2}{\mathfrak{s}(k,r(x))^2}dt-\int_{0}^{r(x)}\Ric^{\perp}_{m,Z}(\dot{\gamma}(t),\dot{\gamma}(t))\frac{\mathfrak{s}(k,t)^2}{\mathfrak{s}(k,r(x))^2} dt\\
		&+\frac{1}{\mathfrak{s}(4k,r(x))^2}\int_{0}^{r(x)}(\mathfrak{s}^{\prime}(4k,t)^2-4k \mathfrak{s}(4k,t)^2 )dt\\
		&\leqslant (m-2)\int_{0}^{r(x)}\left(\frac{\mathfrak{s}^{\prime}(k,t)^2}{\mathfrak{s}(k,r(x))^2}-k \frac{\mathfrak{s}(k,t)^2}{\mathfrak{s}(k,r(x))^2}\right)du
		\\
		& +\frac{1}{\mathfrak{s}(4k,r(x))^2}\int_{0}^{r(x)}(\mathfrak{s}^{\prime}(4k,t)^2-4k \mathfrak{s}(4k,t)^2 )dt\\
		&=(m-2)[\frac{\mathfrak{s}^{\prime}(k,r(x))}{\mathfrak{s}(k,r(x))}\frac{\mathfrak{s}^{\prime}(k,r(x))}{\mathfrak{s}(k,r(x))}]^{r(x)}_{t=0}+[\frac{(\mathfrak{s}^{\prime}(k,r(x)))}{\mathfrak{s}(k,r(x))}\frac{(\mathfrak{s}^{\prime}(4k,r(x)))}{\mathfrak{s}(4k,r(x))}]^{r(x)}_{t=0}\\
		&=(m-2)\frac{\mathfrak{s}^{\prime}(k,r(x))}{\mathfrak{s}(k,r(x))}+\frac{\mathfrak{s}^{\prime}(4k,r(x))}{\mathfrak{s}(4k,r(x))}.
		\end{align*}
		
	\end{proof}

	For the proposition below, let $\left\{X_t^{x} \right\}_{t\geqslant 0}, x \in M$ be the diffusion process  with the infinitesimal generator $\mathcal{L}$. Let us define a stopping time $\sigma_p$ by
	\[
	\sigma_p:=\inf \left\{t\geqslant 0 : d_p(X_t)= \frac{\pi}{2\sqrt{k}} \right\}
	\]
	in the K\"ahler case and
	\[
	\sigma_p:=\inf \left\{t\geqslant 0 : d_p(X_t)= \frac{\pi}{2\sqrt{3k}} \right\}
	\]
	in the quaternionic case.
	
	\begin{prop}
		Given a K\"ahler manifold $(M^n,g,J)$ of the complex dimension $n\geqslant 2$, suppose that for some constant $k>0$, $\Ric^{\perp}_{m,Z}\geqslant (2m-2)k$ and $H \geqslant 4k$ hold on the open ball of raidus $\frac{\pi}{2\sqrt{k}}$ centered at $p$. Then $\sigma_p=\infty$ holds $\mathbb{P}_q$-almost surely for any $q\in M\backslash (Cut_p \cup \left\{ p\right\})$ with $d_p(q)<\frac{\pi}{2\sqrt{k}}$.
		\begin{proof}
			By  It\^{o}'s formula for the radial process $d_p(X_t)$ together with Proposition~\ref{prop:Kahler-Laplacian_comparison}, we have
			\begin{align*}
			& d_p(X_t)\leqslant d_p(q)+\sqrt{2}\beta_t+\int_{0}^{t}\mathcal{L}d_p(X_s)ds
			\\
			& \leqslant d_p(q)+\sqrt{2}\beta_t+\int_{0}^{t}(m-2)\frac{\mathfrak{s}^{\prime}(k,r(X_s))}{\mathfrak{s}(k,r(X_s))}+\frac{\mathfrak{s}^{\prime}(4k,r(X_s))}{\mathfrak{s}(4k,r(X_s))}ds
			\end{align*}	
			for $t<\sigma_p$, where $\beta_t$ is a $1$-dimensional standard Brownian motion. Let us define $\rho_t$ as the solution to the following stochastic differential equation
			\begin{equation*}
			d\rho_t=\sqrt{2}d\beta_t+\left((m-2)\frac{\mathfrak{s}^{\prime}(k,r(\rho_t))}{\mathfrak{s}(k,r(\rho_t))}+\frac{\mathfrak{s}^{\prime}(4k,r(\rho_t))}{\mathfrak{s}(4k,r(\rho_t))}\right)dt
			\end{equation*}
			with $\rho_0=d_p(q)$. (see for example \cite[Theorem 3.5.3]{HsuEltonBook}). Thus it suffices to show that $\rho_t$ never hit $ \frac{\pi}{2\sqrt{k}}$. Since
			\begin{equation*}
			\frac{\mathfrak{s}^{\prime}(4k,r(\rho_t))}{\mathfrak{s}(4k,r(\rho_t))}=\frac{1}{u-\frac{\pi}{2\sqrt{k}}}+o(1)
			\end{equation*}
			as $u \uparrow \frac{\pi}{2\sqrt{k}}$ and $m\geqslant 2n \geqslant 2$, a general theory of $1$-dimensional diffusion processes yields that $\rho_t$ cannot hit $ \frac{\pi}{2\sqrt{k}}$ (see e.g. \cite[Proposition 4.2.2]{HsuEltonBook}).
			
		\end{proof}	
		
	\end{prop}
	By using Proposition above, we can easily show the Bonnet-Myers theorem.
	
	\begin{coro}
		Given a K\"ahler manifold $(M^n,g,J)$ of the complex dimension $n\geqslant 2$, suppose that for some constant $k>0$, $\Ric^{\perp}_{m,Z}\geqslant (2m-2)k$ and $H \geqslant 4k$ hold on $M$. Then the diameter of $M$ is less than equal to  $\frac{\pi}{2\sqrt{k}}$.
		\begin{proof}
			Suppose that there are $p,q \in M$ such that $d(p,q)> \frac{\pi}{2\sqrt{k}}$. We may assume that $M$ is compact and that $\Ric^{\perp}_{m,Z}\geqslant (2m-2)k$, holds on the open ball of radius  $\frac{\pi}{2\sqrt{k}}$ centered at $p$ by modifying outside of a ball of large radius. Then there is an open neighborhood $G$ of $q$ such that $d(p,y)> \frac{\pi}{2\sqrt{k}}$ for all $y\in G$. Take $p^{\prime}$ from a small neighborhood of $p$. Then Proposition yields that $\mathbb{P}_{p^{\prime}}[\sigma_p=\infty]=1$. It implies $\mathbb{P}_{p^{\prime}}[X_t \in G]=0$ for any $t>0$. This is absurd since the law of $X_t$ has a strictly positive density with respect to the Riemannian volume measure for $t>0$.
		\end{proof}
	\end{coro}
	
	The similar propositions of the quaternion K\"ahler case can be obtained by repeating the proofs of the K\"ahler case. We follow the structure of the proof  in \cite[Theorem 3.2]{BaudoinYang2020}.
	
	\begin{prop}\label{prop:quaternion-Kahler-Laplacian_comparison}
		Given a quaternionic K\"ahler manifold $(M^n,g,I,J,K)$ of the quaternionic dimension $n\geqslant 2$ and we denote the Riemannian distance on $M$ from $p$ associated with $g$ by $r$ and $Cut_p$ the cut-locus of $p$.  Suppose that for some constant $k>0$, $\Ric^{\perp}_{m,Z}\geqslant (4m-4)k$ and $Q \geqslant 12k$. Let $x\in M \backslash Cut_p \cup \left\{p\right\}$ with $r(x)<\frac{\pi}{2\sqrt{3k}}$. Then
		\begin{equation*}
		\mathcal{L} r(x)\leqslant (m-4)\frac{\mathfrak{s}^{\prime}(k,r(x))}{\mathfrak{s}(k,r(x))}+\frac{\mathfrak{s}^{\prime}(12k,r(x))}{\mathfrak{s}(12k,r(x))},
		\end{equation*}	
		where $\mathfrak{s}(k,t):=\sin{\sqrt{k}t}$.
	\end{prop}
	\begin{proof}
		When $m=4n$, we just need to put $Z \equiv 0$ and proceed the same argument that we provide below. Thus without loss of generality, we will assume $m>4n$.
		
		Let $ p\in M$ and $x\neq p$ which is not in the cut-locus of $p$. Let $\gamma : [0,r(x)] \rightarrow M$ be the unique arclength parameterized geodesic connecting $p$ to $x$. At $x$, we consider an orthonormal frame $\left\{X_1(x),...,X_{4n}(x) \right\}$ such that
		\begin{equation*}
		X_1(x)=\gamma^{\prime}(r(x)), X_2(x)=IX_1(x), X_3(x)=J X_1(x),X_2(x)=K X_1(x).
		\end{equation*}
		Then
		\begin{equation*}
		\mathcal{L}r= \triangle r(x) + Zr(x)=\sum_{i=1}^{4n}\nabla^2 r(X_i(x),X_i(x)) + Zr(x).
		\end{equation*}
		Since $X_1(x)=\gamma^{\prime}(r(x))$, $\nabla^2 r(X_1(x),X_1(x))$ is zero. Now we divide the above sum into three parts: $\sum_{i=2}^{4}\nabla^2 r(X_i(x),X_i(x))$, $\sum_{i=3}^{4n}\nabla^2 r(X_i(x),X_i(x))$, and $Zr(x)$.
		
		For $\sum_{i=2}^{4}\nabla^2 r(X_i(x),X_i(x))$, note that vectors $IX_1,JX_1,KX_1$ might not be parallel along $\gamma$. Denote $X_2,X_3$, and $X_4$ the vector fields along $\gamma$ obtained by parallel transport along $\gamma$ of $X_2(x), X_3(x)$ and $X_4(x)$. Then one can deduce that
		\begin{align*}
		&R(\gamma^{\prime},X_2,X_2,\gamma^{\prime})+R(\gamma^{\prime},X_3,X_3,\gamma^{\prime})+R(\gamma^{\prime},X_4,X_4,\gamma^{\prime})\\
		&=R(\gamma^{\prime},I\gamma^{\prime},I\gamma^{\prime},\gamma^{\prime})+R(\gamma^{\prime},J\gamma^{\prime},J\gamma^{\prime},\gamma^{\prime})+R(\gamma^{\prime},K\gamma^{\prime},K\gamma^{\prime},\gamma^{\prime})=Q(\gamma^{\prime})
		\end{align*}	
		(see  \cite[Theorem 3.2]{BaudoinYang2020}  for more details).
		Define the vector field along $\gamma$ by
		\begin{equation*}
		\tilde{X}_i(\gamma(t))=\frac{\mathfrak{s}(12k,t)}{\mathfrak{s}(12k,r(x))}J\gamma^{\prime}(t),i=2,3,4,
		\end{equation*}
		where $\mathfrak{s}(k,t):=\sin{\sqrt{k}t}$, we obtain by the same computation as in the Proposition~\ref{prop:Kahler-Laplacian_comparison},
		\begin{align*}
		\sum_{i=2}^{4}\nabla^2 r(X_i(x),X_i(x))\leqslant \frac{1}{\mathfrak{s}(12k,r(x))^2}\int_{0}^{r(x)}(\mathfrak{s}^{\prime}(12k,t)^2-12k \mathfrak{s}(4k,t)^2 )dt\\
		\end{align*}
		The rest steps are similar as in the Proposition~\ref{prop:Kahler-Laplacian_comparison}. Consequently,
		\begin{align*}
		\mathcal{L} r(x)
		&=(m-4)\frac{\mathfrak{s}^{\prime}(k,r(x))}{\mathfrak{s}(k,r(x))}+\frac{\mathfrak{s}^{\prime}(12k,r(x))}{\mathfrak{s}(12k,r(x))}.
		\end{align*}
		
	\end{proof}
	
	Combining the proposition above with Proposition~\ref{prop:quaternion-Kahler-Laplacian_comparison}, we obtained:
	
	\begin{prop}
		Given a quaternionic K\"ahler manifold $(M^n,g,I,J,K)$ of the quaternionic dimension $n\geqslant 2$, suppose that for some constant $k>0$, $\Ric^{\perp}_{m,Z}\geqslant (4m-4)k$ and $Q \geqslant 12k$ hold on the open ball of radius $\frac{\pi}{2\sqrt{k}}$ centered at $p$. Then $\sigma_p=\infty$ holds $\mathbb{P}_q$-almost surely for any $q\in M\backslash (Cut_p \cup \left\{ p\right\})$ with $d_p(q)<\frac{\pi}{2\sqrt{3k}}$.
	\end{prop}
	
	\begin{coro}
		Given a quaternionic K\"ahler manifold $(M^n,g,I,J,K)$ of the quaternionic dimension $n\geqslant 2$, suppose that for some constant $k>0$, $\Ric^{\perp}_{m,Z}\geqslant (4m-4)k$ and $Q \geqslant 12k$ hold on $M$. Then the diameter of $M$ is less than equal to  $\frac{\pi}{2\sqrt{3k}}$.
	\end{coro}
	
	\subsection*{Conflicts of interest}
	The corresponding author states that there is no conflict of interest. 
	
	\bibliographystyle{amsplain}
	\bibliography{Kahler}
	
\end{document}